\documentclass[a4paper,twoside,12pt]{article}

\usepackage{amssymb,amsmath,amsthm,latexsym}
\usepackage{amsfonts}
\usepackage{graphicx,caption,subcaption}
\usepackage[pdftex,bookmarks,colorlinks=false]{hyperref}
\usepackage[hmargin=1.25in,vmargin=1.25in]{geometry}
\usepackage[mathscr]{euscript}

\usepackage[auth-lg,affil-sl]{authblk}
\usepackage{float}
\usepackage{calc,tikz}
\usetikzlibrary{arrows,decorations.markings}
\tikzstyle{vertex}=[circle, draw, inner sep=2pt, minimum size=8pt]
\newcommand{\vertex}{\node[vertex]}

\newtheorem{theorem}{Theorem}[section]
\newtheorem{definition}[theorem]{Definition}

\newtheorem{corollary}[theorem]{Corollary}
\newtheorem{proposition}[theorem]{Proposition}

\newcommand{\noi}{\noindent}

\newcommand{\N}{\mathbb{N}}

\newcommand{\J}{\mathbb{J}}
\newcommand{\sJ}{\mathscr{J}}

\title{\textbf{\sc A Note on $J$-Colouring of Jahangir Graphs}}

\author{Johan Kok}
\affil{\small Centre for Studies in Discrete Mathematics\\ Vidya Academy of Science \& Technology \\ Thrissur-680501, Kerala, India.\\{\tt kokkiek2@tshwane.gov.za}}

\author{Sudev Naduvath\footnote{Corresponding Author}}
\affil{\small Department of Mathematics\\ CHRIST (Deemed to be University) \\ Bengaluru-560029 Karnataka, India.\\{\tt sudevnk@gmail.com}}

\date{}

\begin{document}
\maketitle

\begin{abstract}
\noindent In this paper, we discuss $J$-colouring of the family of Jahangir graphs. Note that the family of Jahangir graphs is a wide ranging family of graphs which by a generalised definition includes wheel graphs. We characterise the subset of Jahangir graphs which admit a $J$-colouring. 
\end{abstract}

\noi \textbf{Keywords:} Jahangir graph, $J$-colouring, rainbow neighbourhood, cordial graph.

\vspace{0.25cm}

\noi \textbf{2010 Mathematics Subject Classification:} 05C15, 05C38, 05C75, 05C85.

\section{Introduction}

For general terminology in graph theory, we refer to \cite{1,3,11}. Unless stated otherwise, all graphs $G$ mentioned in this paper are undirected, finite simple and connected graphs.

Recall that if $\mathcal{C}= \{c_1,c_2,c_3,\ldots,c_\ell\}$, $\ell$ sufficiently large, is a set of distinct colours, a \textit{proper vertex colouring} of a graph $G$, is a vertex colouring $\varphi:V(G) \mapsto \mathcal{C}$ of $G$ such that no two distinct adjacent vertices have the same colour. The minimum number of colours in a proper vertex colouring of $G$ is called the \textit{chromatic number} of $G$ and is denoted $\chi(G)$. 

The number of times a colour $c_i$ is allocated to vertices of a graph $G$ is denoted by $\theta(c_i)$ and the colouring $\varphi:v_i \mapsto c_j$ can be written as $\varphi(v_i)=c_j$. Furthermore, we use another notation $\iota$ in the sense that if $\varphi(v_i)=c_j$, then we say that $\iota(v_i)=j$. Unless mentioned otherwise, we also follow the convention that, among the colours in the coloring $\mathcal{C} =\{c_1,c_2,c_3,\ldots,c_{\ell}\}$, $\ell=\chi(G)$, the colour $c_1$ will be assigned to maximum possible number of vertices in $G$, then colour $c_2$ will be assigned to maximum possible number of remaining uncoloured vertices and proceeding like this, at the final step, the colour $c_\ell$ will be given for the remaining uncoloured vertices. This convention is called the \textit{rainbow neighbourhood convention} (see \cite{5}). Such a colouring is called a \textit{$\chi^-$-colouring} of a graph.


\section{$J$-Colouring of Jahangir graphs}

\begin{definition}\label{Defn-2.1}{\rm 
\cite{0,9} The \textit{Jahangir graph} $J_{n,m},\ n\geq 1,\ m \geq 3$, is a graph on $nm + 1$ vertices consisting of the cycle $C_{nm}$ with an additional central vertex say $u$ which is adjacent to cyclically labeled vertices $v_1, v_2, v_3,\ldots,v_m$ such that $d(v_i,v_{i+1}) = n$, $1\leq i \leq m-1$ in $C_{nm}$.
}\end{definition}

Note that Definition \ref{Defn-2.1} allows the wheel graphs $W_{n'+1}$, $n'\geq 3$ to be part of the family of Jahangir graphs in that $W_{n'+1} = J_{1,n'}$. At this point the authors are aware on only one published generalisation of the definition found in \cite{2}. A good application of Jahangir graphs related to transmitting confidential information between nuclear sites is found in \cite{4}.

The closed neighbourhood $N[v]$ of a vertex $v \in V(G)$ which contains at least one coloured vertex of each colour in the chromatic colouring, is called a rainbow neighbourhood. We say that vertex $v$ yields a rainbow neighbourhood. We also say that vertex $u \in N[v]$ belongs to the rainbow neighbourhood $N[v]$.

\begin{definition}\label{Defn-2.2}{\rm 
\cite{10} A maximal proper colouring of a graph $G$ is a Johan colouring denoted, $J$-colouring, if and only if every vertex of $G$ belongs to a rainbow neighbourhood of $G$. The maximum number of colours in a $J$-colouring is denoted by $\sJ(G)$.
}\end{definition}

\begin{definition}\label{Defn-2.3}{\rm
\cite{10} A maximal proper colouring of a graph $G$ is a modified Johan colouring denoted, $J^*$-colouring, if and only if every internal vertex of $G$ belongs to a rainbow neighbourhood of $G$. The maximum number of colours in a $J^*$-colouring is denoted by $\sJ^*(G)$.
}\end{definition}

Note that, if a graph $G$ admits a $J$-colouring it admits a $J^*$-colouring. However, the converse is not always true. Other interesting result on $J$-colouring are found in \cite{6,7}. We recall three important results from \cite{10}.

\begin{proposition}\label{Prop-2.1}
{\rm \cite{10}} Let $P_n$ denotes a path on $n\geq 3$ vertices. Then, $\sJ(P_n)=2$ and $\sJ^*(P_n)=3$.
\end{proposition}

\begin{theorem}
{\rm \cite{10}} A cycle $C_n$ is $J$-colourable if and only if $n\equiv 0\,({\rm mod}\ 2)$ or $n\equiv 0\,({\rm mod}\ 3)$.
\end{theorem}

\begin{corollary}
{\rm \cite{10}} Let $C_n$ be a cycle which admits a $J$-colouring. Then, 
\begin{equation*}
\sJ(C_n)=
\begin{cases}
2 & \text{if}\ n\equiv 0\,({\rm mod}\ 2) \text{and}\ n\not\equiv 0\,({\rm mod}\ 3),\\
3 & \text{if}\ n\equiv 0\,({\rm mod}\ 3).
\end{cases}
\end{equation*}
\end{corollary}

\begin{theorem}\label{Thm-2.3}
{\rm \cite{10}} Let $W_{n+1}=K_1+C_n$ be a wheel graph which admits a $J$-colouring. Then, 
\begin{equation*}
\sJ(W_{n+1})=
\begin{cases}
3 & \text{if}\ n\equiv 0\,({\rm mod}\ 2) \text{and}\ n\not\equiv 0\,({\rm mod}\ 3),\\
4 & \text{if}\ n\equiv 0\,({\rm mod}\ 3).
\end{cases}
\end{equation*}
\end{theorem}

It immediately follows that Theorem \ref{Thm-2.3} applies directly to Jahangir graphs $J_{1,m} = W_{m+1}$. We now define a needed subset of integers.

\begin{definition}\label{Defn-2.4}{\rm
The \textit{$J$-Jahangir set of integers}, denoted by $\J$, is defined as $\J = \{a_j:a_0 = 1$, $a_i = a_{i-1} +3$, $i \geq 1\}$.
}\end{definition}

We now present our main result which characterises the subset of Jahangir graphs which admit a $J$-colouring.

\begin{theorem}\label{Thm-2.5}
A Jahangir graph $J_{n,m}$ admits a $J$-colouring if and only if $m$ is even and $d(v_i,v_{i+1})= n \in \J$, $1 \leq i\leq m-1$, where $v_j$ are the vertices of the cycle in $J_{n,m}$.
\end{theorem}
\begin{proof}
Consider a cycle $C_m$ with $m$ even. Let the consecutive cycle vertices $v_1,v_2,v_3,\ldots, v_m$ be coloured say, $c_1,c_2,c_1,\ldots,c_1,c_2$. For the wheel $W_{m+1}$ and the central vertex $u$ having colour $c_3$, we have a $J$-colouring for the Jahangir graph $J_{1,m}$. The converse part which states that a Jahangir graph that admits a $J$-colouring as described has $m$ is even and $d(v_i,v_{i+1}) = 1 \in \J$, $1 \leq i\leq m-1$, follows immediately.

Now, if for all edges of $C_m$ we insert $n-1$, $n \in \J$ additional vertices, then $d(v_i,v_{i+1}) \in \J$. From Proposition \ref{Prop-2.1} in respect of the modified $J$-colouring of paths it follows that, if the vertices of a $v_iv_{i+1}$-path from $i=1,2,3,\ldots,m-1$ and those of the $v_mv_1$-path are assigned colours consecutively  $c_1,\underbrace{c_2,c_3,c_1},\ldots,\underbrace{c_2,c_3,c_1},\\ c_2,\underbrace{c_1,c_3,c_2},\ldots,\underbrace{c_1,c_3,c_2}c_1,\ldots,\underbrace{c_2,c_3,c_1},\ldots,\underbrace{c_2,c_3,c_1},c_2,\underbrace{c_1,c_3,c_2},\ldots,\underbrace{c_1,c_3,c_2}$, then we clearly have a $J$-colouring for the Jahangir graph, $J_{n,m}$. Also, $d(v_i,v_{i+1}) \in \J$, $1 \leq i\leq m-1$, where $v_j$ are the vertices defined in Definition \ref{Defn-2.1}. The converse part which states that a Jahangir graph that admits a $J$-colouring as described has $m$ is even and $d(v_i,v_{i+1})= n \in \J$, $1 \leq i\leq m-1$, follows through mathematical induction.

Now, let $m$ be odd. If $m\equiv 0\,({\rm mod}\ 3)$ the cycle $C_m$ and without loss of generality, has the consecutive vertex colours following the pattern say, $c_1,c_2,c_3$ or $c_2,c_3,c_1$ or $c_3,c_1,c_2$ (possibly in the reverse). In the Jahangir graph $J_{1,m}$ the central vertex is coloured say $c_4$. From Proposition \ref{Prop-2.1}, it follows that any $v_iv_{i+1}$-path and the $v_mv_1$-path has $\sJ^* = 3$ and hence no constant number of additional vertices can be inserted in each edge of $C_m$ to construct a $J$-colouring.

If $m$ is odd and $n\not\equiv 0\,({\rm mod}\ 3)$, the cycle $C_m$ does not admit a $J$-colouring. Therefore, the Jahangir graph $J_{1,m}$ cannot either. 
\end{proof}

Figure \ref{fig:J46} illustrates a $J$-colouring for the Jahangir graph $J_{4,6}$. Here, note that $a_1=4 \in \J$.

\begin{figure}[h!]
	\centering
	\begin{tikzpicture}[auto,node distance=2cm,
	thick,main node/.style={circle,draw,font=\sffamily\Large\bfseries}]
	\vertex (0) at (0:0) [label=below:$u$]{$c_3$};
	\vertex (1) at (0:4) [label=right:$v_{1}$]{$c_1$};
	\vertex (2) at (345:4) []{$c_2$};
	\vertex (3) at (330:4) []{$c_3$};
	\vertex (4) at (315:4) []{$c_1$};
	\vertex (5) at (300:4) [label=below:$v_{2}$]{$c_2$};
	\vertex (6) at (285:4) []{$c_1$};
	\vertex (7) at (270:4) []{$c_3$};
	\vertex (8) at (255:4) []{$c_2$};
	\vertex (9) at (240:4) [label=below:$v_{3}$]{$c_1$};
	\vertex (10) at (225:4) []{$c_2$};
	\vertex (11) at (210:4) []{$c_3$};
	\vertex (12) at (195:4) []{$c_1$};
	\vertex (13) at (180:4) [label=left:$v_{4}$]{$c_2$};
	\vertex (14) at (165:4) []{$c_1$};
	\vertex (15) at (150:4) []{$c_3$};
	\vertex (16) at (135:4) []{$c_2$};
	\vertex (17) at (120:4) [label=above:$v_{5}$]{$c_1$};
	\vertex (18) at (105:4) []{$c_2$};
	\vertex (19) at (90:4) []{$c_3$};
	\vertex (20) at (75:4) []{$c_1$};
	\vertex (21) at (60:4) [label=above:$v_{6}$]{$c_2$};
	\vertex (22) at (45:4) []{$c_1$};
	\vertex (23) at (30:4) []{$c_3$};
	\vertex (24) at (15:4) []{$c_2$};
	\path 
	(0) edge (1)
	(0) edge (5)
	(0) edge (9)
	(0) edge (13)
	(0) edge (17)
	(0) edge (21)
	(1) edge (2)
	(2) edge (3)
	(3) edge (4)
	(4) edge (5)
	(5) edge (6)
	(6) edge (7)
	(7) edge (8)
	(8) edge (9)
	(9) edge (10)
	(10) edge (11)
	(11) edge (12)
	(12) edge (13)
	(13) edge (14)
	(14) edge (15)
	(15) edge (16)
	(16) edge (17)
	(17) edge (18)
	(18) edge (19)
	(19) edge (20)
	(20) edge (21)
	(21) edge (22)
	(22) edge (23)
	(23) edge (24)
	(24) edge (1)
	;
	\end{tikzpicture}
	\caption{Jahangir graph $J_{4,6}$.}
	\label{fig:J46}
\end{figure}
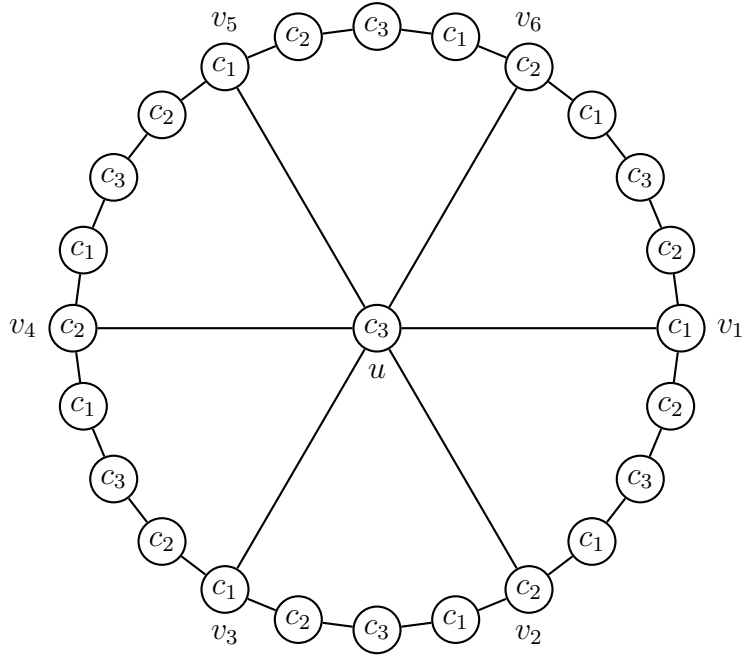

\begin{corollary}
Let $J_{n,m}$ be a Jahangir graph which admits a $J$-colouring. Then, $\chi(J_{n,m})=\sJ(J_{n,m})=3$.
\end{corollary}

\subsection{Special Cordial Jahangir graphs}

Recall that a binary vertex labeling of a graph $G$ is called a \textit{cordial labeling} of $G$ (see \cite{12}) if $|v_f(0) - v_f(1)| \leq 1$ and $|e_f(0) - e_f(1)| \leq 1$. A graph $G$ is cordial if it admits a cordial labeling. In \cite{2}, it is shown that the Jahangir graphs $J_{2k-1,4\ell}$, $k \geq 1$, $\ell \geq 1$ are cordial. In our notation $J_{n,m}$, it means that $n$ belongs to the set of all positive odd integers and $m \in \{4\ell : \ell \in \N\}$. Hence, it follows from Theorem \ref{Thm-2.5} that since $m \in \{4\ell : \ell \in \N\}$ renders $m$ even, that if, $n-1 \equiv 0~(mod~3)$, the Jahangir graphs, $J_{n,m}$ are both cordial and admit a $J$-colouring.

In \cite{2}, it was also shown that the Jahangir graphs $J_{4k-1, 4\ell +2}$, $k \geq 1$, $\ell \geq 1$ are cordial. It means that $n \in \{b_i: b_0 = 3~and~b_{i+1} = b_i + 4, ~ i\geq 1 \}$ and $m \in \{c_j: c_0 = 6~and~c_{j+1} = c_j + 4,~j\geq 1\}$. Hence, since $m$ is even the Jahangir graphs, $J_{n,m}$ are both cordial and admit a $J$-colouring if $n-1 \equiv 0~(mod~3)$.

\section{Conclusion}

In \cite{7}, $J$-colouring of Chithra graphs was considered. This paper considered same for Jahangir graphs. The fact that all graphs admit a $\chi$-colouring but not necessarily a $J$-colouring leaves a wide scope for further research through the various families of graph types. From \cite{7}, we have learnt that, whether or not a graph $G$ admits a $J$-colouring, its Mycielski graph $\mu(G)$ cannot have a $J$-colouring. Hence, there exist a family for which the complete family does not admit a $J$-colouring. Amongst the few families considered thus far, only subsets could be characterised which admit $J$-colouring. It is therefore known that graphs in general do not all admit $J$-colouring. It remains an interesting problem to find a family of graph type for which all admit $J$-colouring, if such exists. The results from \cite{6}, stating that a graph $G$ of order $n$ admits a $J$-colouring if and only if $r_\chi(G)= n$ and alternatively, a graph $G$ admits a $J$-colouring if and only if, in a $\chi^-$-colouring of $G$ each $v \in V(G)$ yields a rainbow neighbourhood, could be key in this search.

It is noted that $\sJ(C_6) = 3$ and $\sJ(\overline{C_6}) = 3$. However, $\sJ(C_9) =3$ whilst $\overline{C_9}$ does not admit a $J$-colouring. It remains open to find graphs $G$ and $\overline{G}$ which both admit a $J$-colouring.

\end{document}